\documentclass[11pt]{amsart}

\usepackage[english]{babel}
\usepackage{amsmath, amsfonts, amssymb, amsthm, epsfig, graphicx, url, hyperref, color, tikz, cancel,subcaption,mathtools,mathrsfs}
\mathtoolsset{showonlyrefs}

\newtheorem{theorem}{Theorem}[section]
\newtheorem{lemma}{Lemma}[section]
\newtheorem{corollary}{Corollary}[section]
\newtheorem{proposition}{Proposition}[section]

\newtheorem{remark}{Remark}[section]

\usepackage[margin=1.55in]{geometry}

\newcounter{theor}

\newtheorem{thm}[theor]{Theorem}

\def\conv{\mathop\mathrm{conv}\nolimits}

\def\s{\mathbb{S}}

\def\E{\mathbb{E}}

\def\R{\mathbb{R}}
\def\N{\mathbb{N}}

\newcommand{\dlat}{\mathrm{d}}

\def \p{\Pi}
\def \pp{\Pi^{\circ}}

\def\esc#1{\left\langle #1\right\rangle}

\numberwithin{equation}{section}

\begin{document}

\title{On stochastic forms of functional isoperimetric inequalities}

\author[F. Marín Sola]{Francisco Marín Sola}
\email{francisco.marin7@um.es}
\address{Department of Engineering and Technology of Computers, area of Applied Mathematics. Universidad de Murcia, Spain.}

\thanks{The author is supported by the grant PID2021-124157NB-I00, funded by MCIN/AEI/10.13039/501100011033/``ERDF A way of making Europe'', as well as by the grant  ``Proyecto financiado por la CARM a través de la convocatoria de Ayudas a proyectos para el desarrollo de investigación científica y técnica por grupos competitivos, incluida en el Programa Regional de Fomento de la Investigación Científica y Técnica (Plan de Actuación 2022) de la Fundación Séneca-Agencia de Ciencia y Tecnología de la Región de Murcia, REF. 21899/PI/22''.}

\subjclass[2020]{Primary 52A38, 52A40, 26B15, 26D15; Secondary 52A20}
\keywords{Isoperimetric inequalities; $p$-concave functions; Rearrangements.}
\maketitle
\begin{abstract}
      We present a probabilistic interpretation of several functional isoperimetric inequalities within the class of $p$-concave functions, building on random models for such functions introduced by P. Pivovarov and J. Rebollo-Bueno. First, we establish a stochastic isoperimetric inequality for a functional extension of the classical quermassintegrals, which yields a Sobolev-type inequality in this random setting as a particular case. Motivated by the latter, we further show that Zhang's affine Sobolev inequality holds in expectation when dealing with these random models of $p$-concave functions. Finally, we confirm that our results recover both their geometric analogues and deterministic counterparts. As a consequence of the latter, we establish a generalization of Zhang's affine Sobolev inequality restricted to $p$-concave functions in the context of convex measures.
\end{abstract}
\section{Introduction}
A fundamental isoperimetric principle in convex geometry states that, among all convex bodies (compact convex sets with non-empty interior) of a given volume, Euclidean balls have the smallest $i$-th quermassintegral for every $i\in\{1,\dots,n-1\}$. More precisely, if $K\subset\R^n$ is a convex body,
\begin{equation}\label{e:classical_isoperimetric}
    W_i(K) \geq W_i(B_K),
\end{equation}
where $W_i(\cdot)$ denotes the $i$-th quermassintegral (see Section \ref{s:background} for a precise definition) and $B_K$ is a Euclidean ball with the same volume as $K$. When $i=1$, \eqref{e:classical_isoperimetric} is nothing but the classical isoperimetric inequality for convex bodies. For context and background the reader may check the excellent books by R. Gardner \cite{gardner_book}, P. Gruber \cite{G07} and R. Schneider \cite{Sch2}.  

Inspired by the  work of H. Groemer \cite{Groemer73} on the expected volume of random polytopes, the isoperimetric-type inequality \eqref{e:classical_isoperimetric} was shown to hold, by R. E. Pfiefer in \cite{Pfiefer82} and by M. Hartzoulaki and G. Paouris in \cite{HarPao}, in expectation when dealing with random convex sets (see also \cite{HarPao}). In a more detailed manner, for any collection of independent random vectors $\{X_k\}_{k=1}^N$ uniformly distributed in $K$,  they proved that 
\begin{align}\label{e:groemer_quermass}
    \E\Bigl[W_i\bigl(\conv\{X_1,\dots,X_N\}\bigr)\Bigl] \geq \E\Bigl[W_i\bigl(\conv\{X^*_1,\dots,X^*_N\}\bigr)\Bigl]
\end{align}
for every $i\in\{0,1,\dots,n-1\}$, where $\{X_k^*\}_{k=1}^N$ are independent random vectors uniformly distributed in $B_K$. Note that \eqref{e:groemer_quermass} recovers \eqref{e:classical_isoperimetric} as $N\to +\infty$. 

A systematic study of stochastic dominance phenomena in isoperimetric-type inequalities was initiated in \cite{PP12} (and further developed in \cite{PP13-2,CEFPP15,PP17-2,APPS24,DPP16}) where, among other findings, it was proven that convexity (or concavity) properties of certain functionals lead to a stochastic dominance in the associated inequalities. More recently, a similar program focused on stochastic analogues of functional inequalities in the realm of $p$-concave functions was started by P. Pivovarov and J. Rebollo-Bueno in \cite{PRB2,PRB}. The aim of this paper is to continue their program by presenting several stochastic analogues of functional isoperimetric inequalities for $p$-concave functions.

In this regard, we will consider a generalization of the classical quermassintegrals to the setting of $p$-concave functions, independently introduced by S. Bobkov, A. Colesanti, and I. Fragalà \cite{BCF14}, and by V. Milman and L. Rotem \cite{MilRot13} (see Section~\ref{s:background}). Among other related results, the following functional extension of \eqref{e:classical_isoperimetric} was proved in \cite{BCF14,MilRot13}. 
\begin{thm}\label{t:func_iso_quermass}
    Let $f: \R^n \to \R_+$ be an integrable $p$-concave function, with $p\in \R \cup \{\pm\infty\}$. Then, for every $i\in\{1,\dots,n-1\}$,
    \begin{equation}\label{e:func_iso_quermass}
    W_i(f) \geq W_i(f^*),
    \end{equation}
    where $f^*$ denotes the symmetric decreasing rearrangement of $f$.
\end{thm}

In light of \eqref{e:groemer_quermass}, it is natural to wonder about a stochastic analogue of \eqref{e:func_iso_quermass}. Our first main result provides such an analogue. To state it precisely, we first introduce the stochastic framework developed by P. Pivovarov and J. Rebollo-Bueno. Let $N\geq n+1$. Given an integrable $p$-concave function $f:\R^n \to \R_+$, with $p \in \R\cup\{+\infty\}$, and $\{(X_k,Z_k)\}_{k=1}^N$ independent random vectors uniformly distributed w.r.t. the Lebesgue measure on the hypograph of $f$, i.e., uniformly sampled in
$$
\text{hyp}(f) = \{(x,z)\in \R^n \times \R_+
: f(x) \geq z\}
$$
(for the sake of simplicity, in the following we will simply say distributed w.r.t. $f$). The so-called \emph{stochastic model} of $f$ is built upon the random sets $E_{N,p},H_{N,p}\subset\R^n\times\R_+$ given by \begin{equation}\label{e:stoch_model_convexhull_p<=0}
    E_{N,p} =  \begin{cases} 
      \conv\{R_{Z_1^p}(X_1),\dots,R_{Z_N^p}(X_N)\} & \mathrm{if}\quad p< 0, \\
       \conv\{R_{-\log Z_1}(X_1),\dots,R_{-\log Z_N}(X_N)\}& \mathrm{if}\quad p=0 
   \end{cases}
\end{equation}
and
\begin{equation}\label{e:stoch_model_convexhullp>0}
     H_{N,p} = \conv\{\widetilde{R}_{Z_1^p}(X_1),\dots,\widetilde{R}_{Z_N^p}(X_N)\}
\end{equation}
if $p>0$, where for any given  $(x,z)\in \R^n\times \R_+$ the sets $R_{z}(x)$ and $\widetilde{R}_{z}(x)$ are defined as 
$$
R_{z}(x) =\{(x,r) : r\geq z\} \quad \mathrm{and} \quad \widetilde{R}_z(x) = \{(x,r): r\leq z\}.
$$ 

Specifically, the stochastic model of $f$ constructed from the random vectors $\{(X_k,Z_k)\}_{k=1}^N$, denoted as $\Phi^N_{(X_k,Z_k)}$, is the $p$-concave function defined via epigraphs and hypographs as follows
\begin{align}\label{e:random_epi-hyp}
    &\mathrm{epi}\bigl((\Phi^N_{(X_k,Z_k)})^p\bigr) =E_{N,p}  \quad \mathrm{if} \quad p<0,\\
    &\mathrm{epi}(-\log \Phi^N_{(X_k,Z_k)}) =E_{N,p} \quad \mathrm{if} \quad p=0,\\
    & \mathrm{hyp}\bigl((\Phi^N_{(X_k,Z_k)})^p\bigr) = H_{N,p}\quad \mathrm{if} \quad p>0.
\end{align} 
Equivalently, defining
\begin{equation}\label{e:stoch_model_convexhull_2_p<=0}
    e_{N,p} =  \begin{cases} 
      \conv\{(X_1,Z_1^p),\dots,(X_N,Z_N^p)\} & \mathrm{if}\quad p< 0, \\
       \conv\{(X_1,-\log Z_1),\dots,(X_N,-\log Z_N)\}& \mathrm{if}\quad p=0,
   \end{cases}
\end{equation}
and
\begin{equation}\label{e:stoch_model_convexhull_2_p>0}
     h_{N,p} = \conv\{(X_1,Z_1^p),\dots,(X_N,Z_N^p)\}
\end{equation}
when $p>0$, we have that
\begin{align}\label{e:stoch_model}
    \Phi^N_{(X_k,Z_k)}(x) =\begin{cases}
    \inf\{z^{1/p} : (x,z) \in e_{N,p}\}, &\mathrm{if}\quad p<0\\
    \sup\{e^z : (x,z)\in e_{N,p}\}, & \mathrm{if} \quad p=0\\
    \sup\{z^{1/p} : (x,z) \in h_{N,p}\} & \mathrm{if} \quad p>0.
\end{cases} 
\end{align}

\begin{figure}[htbp]
    \centering
    \begin{subfigure}[b]{0.45\textwidth}
        \centering
        \includegraphics[width=\textwidth]{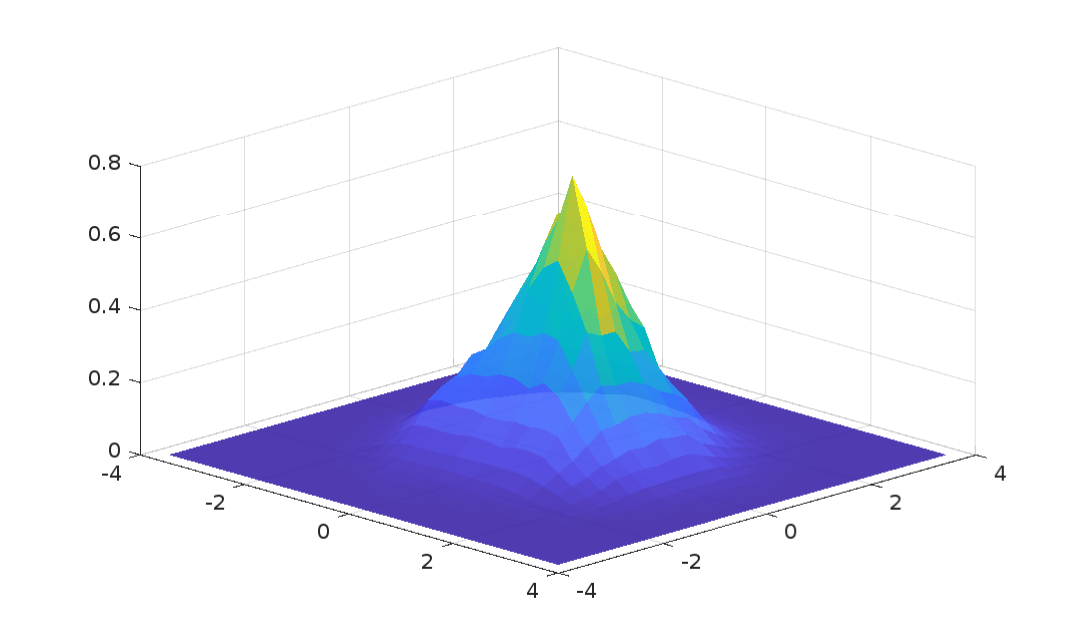}
        \caption{$N=1000$}
        \label{fig:image1}
    \end{subfigure}
    \hfill
    \begin{subfigure}[b]{0.45\textwidth}
        \centering
        \includegraphics[width=\textwidth]{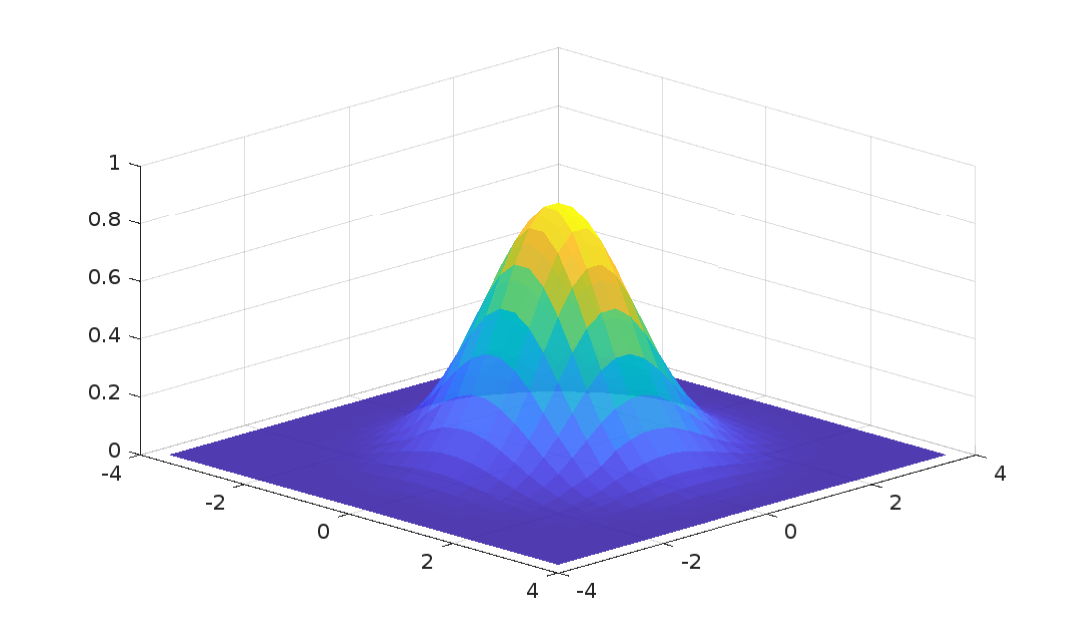}
        \caption{$N=100000$}
        \label{fig:image2}
    \end{subfigure}
    \caption{Comparative of two stochastic models of a  bidimensional  centered gaussian}
    \label{fig:two_images}
\end{figure}
Our first main result reads as follows.
\begin{theorem}\label{t:stoch_iso_quermass}
      Let $f:\R^n \to \R_+$ be an integrable $p$-concave function and $\{(X_k,Z_k)\}_{k=1}^N$ be independent random vectors distributed w.r.t. $f$. Then, for every $i\in\{0,\dots, n-1\}$
     $$
     \E\left[W_i\Bigl(\Phi^N_{(X_k,Z_k)}\Bigr)\right]\geq \E\left[W_i\Bigl(\Phi^N_{(X^*_k,Z^*_k)}\Bigr)\right],
     $$
       where $\Phi^N_{(\cdot,\cdot)} : \R^n\to\R_+$ is a stochastic model  and $\{(X_k^*,Z_k^*)\}_{k=1}^N$ are independent random vectors distributed w.r.t. the symmetric decreasing rearrangement of $f$.
\end{theorem}

The case of $i=1$  in Theorem \ref{t:func_iso_quermass} is precisely the Sobolev inequality in $\R^n$ for functions of bounded variation (see e.g. \cite[Chapter~3]{AFP} for background), which is known to hold without the $p$-concavity assumption. In more detail, under the hypotheses of Theorem \ref{t:func_iso_quermass}, we have  that
\begin{equation}\label{e:classical_Sobolev}
    \mathrm{Per}(f) \geq \mathrm{Per}(f^*),
\end{equation}
where $\mathrm{Per}(f) = nW_1(f)$. Thus, the case of $i=1$ in Theorem \ref{t:stoch_iso_quermass} yields a stochastic form of the classical Sobolev inequality in the class of $p$-concave functions.
\begin{corollary}\label{c:stoc_sobolev}
    Let $f:\R^n \to \R_+$ be an integrable $p$-concave function and $\{(X_k,Z_k)\}_{k=1}^N$ be independent random vectors distributed w.r.t. $f$. Then
    $$
    \E\left[\mathrm{Per}\Bigl(\Phi^N_{(X_k,Z_k)}\Bigr)\right]\geq \E\left[\mathrm{Per}\Bigl(\Phi^N_{(X^*_k,Z^*_k)}\Bigr)\right],
    $$
     where $\Phi^N_{(\cdot,\cdot)} : \R^n\to\R_+$ is a stochastic model  and $\{(X_k^*,Z_k^*)\}_{k=1}^N$ are independent random vectors distributed w.r.t. the symmetric decreasing rearrangement of $f$.
\end{corollary}

An affine strengthening of \eqref{e:classical_Sobolev} is Zhang's affine Sobolev inequality, established by G. Zhang in \cite{Z99} for $C^1$ functions and later extended to functions of bounded variation by T. Wang in \cite{TW12}. G. Zhang proved that this inequality is equivalent to the Petty projection inequality \cite{CMP71} for convex bodies:
\begin{equation}\label{e:classical_petty}
|\pp K|\leq |\pp B_K|,
\end{equation}
where $\pp$ denotes the polar projection body operator and $|\cdot|$ stands for the usual Lebesgue measure (or \emph{volume}). Very recently, G. Paouris, P. Pivovarov and K. Tatarko \cite{PPT} have shown that Petty's inequality \eqref{e:classical_petty} also holds in a stochastic form. This result, combined with Corollary \ref{c:stoc_sobolev}, motivates our second main objective: to establish that, within the class of $p$-concave functions,  Zhang's affine Sobolev inequality holds in expectation for the stochastic models defined in \eqref{e:stoch_model}.

To achieve this goal, we employ the notion of functional polar projection bodies, which originates from ideas of E. Lutwak, D. Yang and G. Zhang \cite{LYZ06,LYZ02} (see also \cite{ML12,HL22,HL24}). This concept was also formally introduced in \cite{CLM17_2,AGJV18} (see Section \ref{s:background} for a precise definition). Within this framework, Zhang's affine Sobolev inequality takes the following form. Note that, as mentioned, this inequality is known to hold without the $p$-concavity assumption.
\begin{thm}\label{t:affine_sobolev}
    Let $f : \R^n \to \R_+$ be an integrable $p$-concave function. Then 
    $$
    |\pp f| \leq |\pp f^*|,
    $$
     where $f^*$ denotes the symmetric decreasing rearrangement of $f$.
\end{thm}
Our second main result reads as follows. Note that rather than the Lebesgue measure we consider any rotationally invariant \emph{convex measure} on $\R^n$, i.e., an absolutely continuous measure with (rotationally invariant) $(-1/n)$-concave density. 
\begin{theorem}\label{t:stoc_affine_sobolev}
    Let $f:\R^n \to \R_+$ be an integrable $p$-concave function and $\{(X_k,Z_k)\}_{k=1}^N$ be independent random vectors distributed w.r.t. $f$.  Then, for any rotationally invariant convex measure $\nu$ on $\R^n$,
    $$
    \E\left[\nu\bigl(\pp \Phi^N_{(X_k,Z_k)}\bigr)\right] \leq \E\left[\nu\bigl(\pp \Phi^N_{(X^*_k,Z^*_k)}\bigr)\right],
    $$
    where $\Phi^N_{(\cdot,\cdot)} : \R^n\to\R_+$ is a stochastic model and $\{(X_k^*,Z_k^*)\}_{k=1}^N$ are independent random vectors distributed w.r.t. the symmetric decreasing rearrangement of $f$.
\end{theorem}

We would like to remark that the proof of Theorem \ref{t:stoc_affine_sobolev} follows the approach of \cite{PPT}, which itself builds upon the ideas presented in \cite[Section~8.2]{MilYe}. For our purposes, we employ an interpretation of the functional polar projection bodies using the functional mixed volumes of V. Milman and L. Rotem \cite{MilRot13}. Furthermore, Theorem~\ref{t:stoch_iso_quermass} recovers the geometric counterparts established by R. E. Pfiefer \cite{Pfiefer82} and by M.~Hartzoulaki and G.~Paouris \cite{HarPao}, while Theorem~\ref{t:stoc_affine_sobolev} recovers that of G. Paouris, P. Pivovarov, and K. Tatarko \cite{PPT} (see Remark~\ref{r:geometrical}).

Given these connections, it is natural to ask whether Theorems \ref{t:func_iso_quermass} and \ref{t:affine_sobolev} can also be recovered from their stochastic versions. We conclude by addressing this question, noting first that the stochastic models are well-defined for integrable functions, without requiring any concavity assumption. However, the resulting random sets (see \eqref{e:random_epi-hyp}) need not approximate the original function unless it is $p$-concave. Hence, the $p$-concavity hypothesis in Theorems~\ref{t:stoch_iso_quermass} and~\ref{t:stoc_affine_sobolev} is imposed to ensure that the deterministic inequalities can be recovered from their stochastic counterparts.

As a consequence of the latter investigation, we obtain the following generalization of Theorem \ref{t:affine_sobolev}.
\begin{theorem}
    Let $f : \R^n \to \R_+$ be an integrable $p$-concave function. Then, for any rotationally invariant convex measure $\nu$ on $\R^n$,
    \begin{equation}\label{e:affine_Sobolev_convex_measures}
        \nu(\pp f) \leq \nu(\pp f^*),
    \end{equation}
     where $f^*$ denotes the symmetric decreasing rearrangement of $f$.
\end{theorem}
Although a proof can be given using similar ideas to those of Theorem \ref{t:stoc_affine_sobolev}, inequality \eqref{e:affine_Sobolev_convex_measures} is, to the best of our knowledge, new.

This paper is organized as follows. In Section \ref{s:background} we collect some preliminaries, background material and several tools that we will use later on, while Section \ref{s:main} is devoted to the proofs of Theorems \ref{t:stoch_iso_quermass} and \ref{t:stoc_affine_sobolev}. Finally, we show in Section \ref{s:deterministic} how to recover Theorems \ref{t:func_iso_quermass} and \ref{t:affine_sobolev} from their stochastic versions.

\section{Preliminaries}\label{s:background}
We work in the $n$-dimensional Euclidean space $\R^n$ endowed with the standard inner product $\esc{\cdot,\cdot}$. We write $e_i$ for the $i$-th canonical unit vector, and $x_{i}$ is used for the $i$-th coordinate of a vector in such a space. We denote by $B^n_2$ the $n$-dimensional Euclidean (closed) unit ball and by $\s^{n-1}$ its boundary. We use $+$ for the Minkowski sum, i.e., $A+B=\{a+b:\, a\in A, \, b\in B\}$ for any non-empty sets $A, B\subset\R^n$. Given any set $M \subset \R^n$, $\chi_{_{M}}$ denotes its characteristic function. Moreover, for any unit direction $u\in\s^{n-1}$, $u^\perp$, $P_{u^\perp}$ and $R_u$ are used for a hyperplane with normal vector $u$, the orthogonal projection onto it, and the reflection map about it, respectively. Along the paper, the $k$-dimensional Lebesgue measure of a measurable set is denoted by $|\cdot|_k$ and we will omit the index $k$ when it is equal to the dimension of the ambient space; furthermore, as usual, integrating $\dlat x$ will stand the integration with respect to the Lebesgue measure. 

Let $K\subset\R^n$ be a compact convex set, its support function $h_K$ and polar body $K^\circ$ are given by
$$
h_K(y) = \sup_{x\in K}\esc{x,y} \quad \mathrm{and}\quad K^\circ=\{x \in \R^n : h_K(x)\leq 1\}.
$$
If $K$ contains the origin in its interior, its Minkowski functional is defined as $\|x\|_{K} = \inf \{\lambda > 0 : x \in \lambda K\}$. Note that, in this case, $K = \{x \in \R^n : \|x\|_K \leq 1\}$ and $\|\cdot\|_{K^\circ} = h_K(\cdot)$. We will also recall some definitions; the projection body, $\p K$, of $K$ is the centrally symmetric convex body whose support function is given by $h_{\p K}(u) = |P_{u^\perp}K|_{n-1}$. Thus the polar projection body, denoted as $\pp K$, is merely $(\p K)^\circ$. For convex sets $K,L \subset \R^n$ the Hausdorff distance between them, denoted as $\delta^H(K,L)$ is
$$
\delta^H(K,L) = \inf\{\varepsilon>0: K \subset L + \varepsilon B_2^n, L\subset K + \varepsilon B_2^n\}.
$$
Equivalently,
$$
\delta^H(K,L) = \sup_{u\in \s^{n-1}}|h_K(u)-h_L(u)|.
$$
The following remark on the interplay between polarity and Hausdorff distance is in order; let $\{K_N\}_{N\in\N}\subset\R^n$ be a sequence of centrally symmetric convex bodies such that
$$
K_N \overset{\delta^H}{\to} K
$$
as $N \to +\infty$, where $K\subset\R^n$ is a centrally symmetric convex body. Then 
$$
K^\circ_N \overset{\delta^H}{\to} K^\circ
$$
as $N \to +\infty$.

Let $K_1,...,K_m\subset \R^n$ be convex bodies and $\lambda_1,...,\lambda_m \geq 0$. The volume of their Minkowski sum is given by
$$
|\lambda_1K_1 + \cdot \cdot \cdot+\lambda_mK_m| = \sum_{1\leq i_1,....,i_n\leq m}\lambda_{i_1}\cdot\cdot\cdot \lambda_{i_n}V(K_{i_1},...,K_{i_n}),
$$
where the coefficient $V(K_{i_1},...,K_{i_n})$ is the so-called mixed volume of $n$-tuple $(K_{i_1},...,K_{i_n})$ (see \cite[Section~5]{Sch2} for background and properties). We will use the standard abbreviation $V(K[n-k],L[k])$ when the sets appear, respectively, $n-k$ and $k$ times in the mixed volume with $k\in\{1,...,n-1\}$. In addition, for any $i\in\{1,...,n-1\}$, the $i$-th quermassintegral of $K$,  is denoted by $W_i(K) = V(K[n-i],B_2^n[i])$. Note that, with this terminology, $\|u\|_{\pp K} = nV(K[n-1],[0,u])$ for every $u\in\s^{n-1}$. Moreover, we will use that mixed volumes are continuous w.r.t. the Hausdorff distance.

We will also use the $M$-sum of sets originally introduced by R. Gardner, D. Hug and W. Weil \cite{GHW13}. Let $M\subset \R^N$ and $K_1,\dots,K_N\subset \R^n$, their $M$-combination is defined as
$$
\oplus_M(K_1,\dots,K_N) = \left\{\sum_{i=1}^N m_i x_i : x_i \in K_i, (m_1,\dots,m_N)\in M\right\}.
$$
It was proved in \cite[Theorem~6.1]{GHW13} that, if $K_1,\dots,K_N$ are convex and $M$ is a compact convex set contained in the positive orthant or centrally symmetric, then $\oplus_M(K_1,\dots,K_N)$ is convex.

We remark that the use of the $M$-addition allows us to recover certain random sets that have already been used in the literature when proving stochastic forms of isoperimetric inequalities (see e.g. \cite{PP12} and \cite{PPT}). In this regard, note that if $[x_1,\dots,x_N]$ is the $n\times N$ matrix with columns $x_i$ then
$$
\oplus_C(\{x_1\},\dots,\{x_N\}) = [x_1,\dots,x_N]C.
$$
Moreover, taking $C = \conv\{e_1,\dots,e_N\}$ one gets
$$
\oplus_C(\{x_1\},\dots,\{x_N\}) = \conv\{x_1,\dots,x_N\}.
$$

\subsection{Linear parameter systems}
We will use the notion of a \emph{linear parameter system} introduced by C. A. Rogers and G. C. Shepard in \cite{RS58:2}  (see also \cite[Section~10.4]{Sch2}). Let $K \subset \R^n$ be a convex body, and $u\in\s^{n-1}$ be a unit direction. A linear parameter system is a family of convex bodies $\{K(t)\}_{t\in I}$, where $I\subset\R$ is an interval, that can be represented as
\begin{equation}\label{e:linear_parameter_system}
    K(t) = \conv\{x_j + \lambda_j t u : j \in \mathcal{J}\}.
\end{equation}
Here $\{x_j\}_{j\in\mathcal{J}}$ and $\{\lambda_j\}_{j\in \mathcal{J}}$ are bounded sets in $\R^n$ and $\R$ respectively, and $\mathcal{J}$ is an arbitrary index set. We will also make use of the case in which the index set is a convex body $K$, i.e,
$$
K(t) = \text{conv}\{ x + t \alpha(x)u : x\in K\},
$$
where $\alpha : K\to\R$ is a bounded function. An example of this construction is the following one; let $K\subset \R^n$ be a convex body and $u \in \s^{n-1}$. Then $\{K_u(t)\}_{t\in [-1,1]}$ is the linear parameter system given by 
\begin{equation}\label{e:shadow_system}
    K_u(t)= \{y + su: y\in P_{u^\perp}K, \; s\in [f_u^t(y),-g_u^t(y)]\},
\end{equation}
with
$$
f_u^t(y) = \frac{(1+t)f_u(y) + (1-t)g_u(y)}{2}
$$
and 
$$
 g_u^t(y) = \frac{(1-t)f_u(y) + (1+t)g_u(y)}{2},
$$
where $f_u,g_u : P_{u^\perp}K \to \R$ are convex functions such that
$$
K=\{ y + su: y\in P_{u^\perp}K, \; s\in [f_u(y),-g_u(y)]\}.
$$
Note that $\{K_u(t)\}_{t\in [-1,1]}$ interpolates continuously between $K_u(1) = K$, $K_u(-1) = R_u K$ and $K_u(0) = S_{u^\perp}K$, where $S_{u^\perp}K$ is the \emph{Steiner symmetral} of $K$ in the direction $u$. 

A fundamental property of linear parameter systems is their convexity under mixed volumes: an elegant proof can be found in \cite[Theorem~10.4.1]{Sch2}.
\begin{theorem}\label{t:convexity_Mixed_Volumes)}
    Let $\{K_i(t)\}_{t\in I}$, with $i = 1,...,n$, be linear parameter systems in the direction $u\in \s^{n-1}$. Then $t \mapsto V\bigl(K_1(t),...,\,K_n(t)\bigr)$ is convex.
\end{theorem}
As a corollary one has the following.
\begin{corollary}\label{c:shepard_convexity}
    Let $\{K(t)\}_{t\in I }$ be a linear parameter system. Then for every $i=0,\dots,n-1$ $t\mapsto W_i\bigl(K(t)\bigr)$ is convex.
\end{corollary}

\subsection{Functional setting}
We recall that a function $\varphi:\R^n\longrightarrow\R_{\geq0}$ is
$p$-concave, for $p\in\R\cup\{\pm\infty\}$, if
\begin{equation*}\label{e:p-concavecondition}
\varphi\bigl((1-\lambda)x+\lambda y\bigr)\geq
\bigl((1-\lambda)\varphi(x)^p+\lambda \varphi(y)^p\bigr)^{1/p}
\end{equation*}
for all $x,y\in\R^n$ such that $\varphi(x)\varphi(y)>0$ and any $\lambda\in(0,1)$, where the cases $p=0$, $p=\infty$ and $p=-\infty$ must be understood as the corresponding expressions that are obtained by continuity, namely, the geometric mean, the maximum and the minimum (of $\varphi(x)$ and $\varphi(y)$), respectively. A $0$-concave function is usually called \emph{log-concave} whereas a $(-\infty)$-concave function is referred to as \emph{quasi-concave}.
Moreover, Jensen's inequality implies that a $q$-concave function is also $p$-concave, whenever $q>p$.

In the following we will assume that $p$-concave functions are upper semicontinuous. Indeed, otherwise we may replace the function by its upper closure, which is determined via the closure of the superlevel sets (see \cite[Theorem~1.6]{RW97}), and thus their Lebesgue measure is preserved. In particular, this will imply that the maximum of such a function is attained. Additionally, we will always consider integrable $p$-concave functions. The latter, together with upper semicontinuity, imply that in our framework the superlevel sets of a $p$-concave function are convex bodies.

Let $s\in[-\infty,1]$ and $\nu$ be a Borel measure in $\R^n$. Then $\nu$ is $s$-concave if
$$
\nu\bigl((1-\lambda) A + \lambda B\bigr)\geq \bigl((1-\lambda)\nu(A)^s + \lambda \nu(B)^s\bigr)^{1/s}
$$
When $s = -\infty$, the measure is usually referred to as \emph{convex}. Following Borell's characterization \cite{Borell}, an absolutely continuous measure $\nu$ in $\R^n$ with density $\varphi$ is $s$-concave if and only if $\varphi$ is $p$-concave with $p=s/(1-ns)$ (Note that Jensen's inequality implies that convex measures are the largest class among $s$-concave ones). The latter can be deduced from the following result, originally proved in
\cite{Borell} and \cite{BL} (see also \cite{G} for a detailed presentation).
\begin{theorem}[The Borell-Brascamp-Lieb inequality]\label{t:BBL}
Let $\lambda\in(0,1)$. Let $-1/n\leq p\leq\infty$ and let $f,g,h:\R^n\longrightarrow\R_{\geq0}$ be measurable functions, with positive integrals, such that
\begin{equation*}
h\bigl((1-\lambda)x + \lambda y\bigr)\geq \bigl((1-\lambda)f(x)^p + \lambda g(y)^p\bigr)^{1/p}
\end{equation*}
for all $x,y\in\R^n$ such that $f(x)g(y)>0$. Then
\begin{equation}\label{e:BBL}
\int_{\R^n}h(x)\,\dlat x\geq \left((1-\lambda)\left(\int_{\R^n}f(x)\,\dlat x\right)^q+\lambda\left(\int_{\R^n}g(x)\,\dlat x\right)^{q}\right)^{1/q},
\end{equation}
where $q=p/(np+1)$.
\end{theorem}

Another consequence of the Borell-Brascamp-Lieb inequality is the following. 
\begin{corollary}\label{c:marginals}
    Let $f:\R^n\times\R^d$ be a $p$-concave function, with $p\geq -1/n$. Then
    $$
    F(y) = \int_{\R^n}f(x,y)\,\dlat x
    $$
    is $p/(np+1)$-concave.
\end{corollary}

We will also recall the notion of symmetric decreasing rearrangement.  We essentially follow \cite{Burchard} (see also \cite[Chapter~3]{LiebAnalysis}). Let $A\subset\R^n$ be a measurable set with finite volume. Its symmetric rearrangement $A^*$ is a Euclidean open ball with the same volume as $A$. Let now $f:\R^n \to \R_+$ be an integrable function. Using its \emph{layer-cake representation} 
$$
f(x) = \int_0^{+\infty} \chi_{_{\{f>t\}}}(x)\,\dlat t,
$$
the \emph{symmetric decreasing rearrangement} of $f$, denoted as $f^*$, is given by
$$
f^*(x) = \int_0^{+\infty} \chi_{_{\{f>t\}^*}}(x)\,\dlat t.
$$
Note that, for any given $t\in\R_+$, $\{f\geq t\}$ is used to denote the superlevel set $\{x\in \R^n : f(x) \geq t\}$. We also emphasize that $f^*$ is radially symmetric and decreasing. In addition, it preserves the volume of the superlevel sets, i.e., $|\{f > t\}| = |\{f^*>t\}|$ for all $t\in\R_+$. Moreover, one has $\|f^*\|_p = \|f\|_p$ for each $1\leq p\leq +\infty$, where $\|\cdot\|_p$ is the usual $L_p(\R^n)$-norm. It is also worth mentioning that the symmetric decreasing rearrangement preserves log-concavity (see \cite[Proposition~24]{MilRot13}).

For any given $u\in \s^{n-1}$ the Steiner symmetral of $f$ with respect to $u^\perp$, denoted as $f^u$ is the function given by 
$$
f^u(x) = \int_0^{+\infty} \chi_{_{\bigl\{S_{u^\perp}\{f>t\}\bigr\}}}(x)\,\dlat t.
$$
Equivalently, $f^u$ is obtained rearranging $f$ along every line parallel to $u$, i.e., for every $y\in u^\perp$, taking $h(t) = f(y + tu)$, we have that $f^u(y+ tu) = h^*(t)$. It is proved in \cite{BLL74} (see also \cite[Chapter~14]{BSConv11}) that for every measurable function $f:\R^n \to \R_+$ with compact support there exists a sequence of the form $f_0 = f$ and $f_{n+1} = f_{n}^u$ for some $u \in \s^{n-1}$, which converges in the $L_1(\R^n)$-norm to $f^*$.

A key result involving symmetric decreasing rearrangements, which will be central to this work, is Christ's version \cite{Christ84} of the Rogers-Brascamp-Lieb-Luttinger inequality. As shown in \cite{PP12}, this theorem is a powerful tool for proving stochastic isoperimetric inequalities. We state it here for the reader's convenience.
\begin{theorem}[\cite{Christ84}]\label{t:Christ_BL}
  Let $f_1,...,f_N : \R^n \to  \R_+$ be integrable functions and let $F : ( \R^n)^N \to \R_+$. Suppose that $F$ satisfies that, for any $u \in \s^{n-1}$ and $ y =(y_1,...,y_N) \in (u^\perp)^N$, the function $F_{u,y}:\R^N \to \R_+$ defined by $F_{u,y}(t_1,...,t_N) = F(y_1 + t_1 u, ...,y_N+t_N u)$ is even and quasi-concave. Then 
    \begin{equation}\label{e:Christ_BLL}
         \int_{(\R^n)^N}F(x_1,...,x_N)\prod_{i=1}^N f_i(x_i) \, \dlat  x \leq \int_{(\R^n)^N}F(x_1,...,x_N)\prod_{i=1}^N f^*_i(x_i) \, \dlat  x,
    \end{equation}
     where $\dlat  x$ stands for $\dlat x_1\cdots\dlat x_N$.
\end{theorem}

Indeed, under the assumptions of Theorem \ref{t:Christ_BL} one can check that
\begin{align}\label{e:BBL_Steiner}
    \int_{(\R^n)^N}F(x_1,...,x_N)\prod_{i=1}^N f_i(x_i) \, \dlat  x \leq \int_{(\R^n)^N}F(x_1,...,x_N)\prod_{i=1}^N f_i^u(x_i) \, \dlat x,
\end{align}
 which implies \eqref{e:Christ_BLL} after a sequence of Steiner symmetrizations with respect to suitable directions. We refer the reader to \cite[Proposition~3.2]{PP12} for a detailed exposition of the latter. Note that, if $F_{u, y}$ is quasi-convex for any $u \in \s^{n-1}$ and $ y =(y_1,...,y_N) \in (u^\perp)^N$, the inequalities \eqref{e:Christ_BLL} and \eqref{e:BBL_Steiner} are reversed. 

\subsection{Functional mixed volumes}

The notion of  mixed volumes was extended to a functional setting by V. Milman and L. Rotem \cite{MilRot13} and S. Bobkov, A. Colesanti and I. Fragalà \cite{BCF14} independently. Although we essentially work in the class of $p$-concave functions, with $p\in \R\cup\{+\infty\}$, we will introduce this concept for the larger class of quasi-concave functions. Let $f_1, \dots, f_n :\R^n \to \R_+$ be  quasi-concave functions. Their mixed integral or functional mixed volume is defined as
\begin{equation*}
    V(f_1,\dots,f_n)=\int_0^{+\infty} V\bigr(\{f_1 \geq t\},\dots,\{f_n \geq t\}\bigl)\,\dlat t.
\end{equation*}
Following the spirit of the geometrical case, for any given quasi-concave function $f:\R^n \to \R_+$, its $i$-th quermassintegral is 
    \begin{equation*}
        W_i(f) = V(f[n-i],\chi_{_{B_2^n}}[i]) = \int_0^{+\infty} W_i\bigl(\{ f \geq t\}\bigr)\,\dlat t,
    \end{equation*}
    where $i\in\{0,1,...,n-1\}$.

Furthermore, for any $p$-concave function $f:\R^n\to \R_+$ its polar projection body is the centrally symmetric convex set, denoted as $\pp f$, whose Minkowski functional is given by
$$
\|u\|_{\pp f} = n V(f[n-1],\chi_{_{[0,u]}}).
$$
Note that 
$$
\|u\|_{\pp f} = n\int_0^{+\infty}V(\{f\geq t\}[n-1],[0,u])\,\dlat t = \int_0^{+\infty}|P_{u^\perp}\{f\geq t\}|_{n-1}\,\dlat t.
$$
Hence, when $f$ belongs to the Sobolev space $W^{1,1}(\R^n)$ of integrable functions with integrable first-order weak derivatives, one obtains using the co-area formula that
$$
\|u\|_{\pp f} = \frac{1}{2}\int_{\R^n}|\esc{\nabla f(x),u}|\,\dlat x.
$$
From the latter, it is not difficult to check that the inequality
$$
|\pp f| \leq |\pp f^*|,
$$
is equivalent to Zhang's affine Sobolev inequality restricted to the class of $p$-concave functions in $W^{1,1}(\R^n)$.

\section{Main results}\label{s:main}
Given an integrable $p$-concave function $f:\R^n \to \R_+$, with $p \in \R\cup\{+\infty\}$, and let $\{(X_k,Z_k)\}_{k=1}^N$ be independent random vectors uniformly distributed w.r.t. $f$. As pointed out in \cite{PRB2}, it is possible to consider more general forms of the stochastic model introduced in \eqref{e:stoch_model} by using the $M$-sum. Specifically, let $C\subset \R^N$ be a compact convex set contained in the positive orthant. We can define the random convex sets $E_{N,p},H_{N,p}\subset\R^n\times\R_+$ as
\begin{align}\label{e:epi_p<=0}
    E_{N,p} =  \begin{cases} 
      \oplus_C\bigl(R_{Z_1^p}(X_1),\dots,R_{Z_N^p}(X_N)\bigr) & \mathrm{if}\quad p< 0 \\
      \oplus_C\bigl(R_{-\log Z_1}(X_1),\dots,R_{-\log Z_N}(X_N)\bigr) & \mathrm{if}\quad p=0 
   \end{cases}
\end{align}
and 
\begin{align}\label{e:hypo_p>0}
    H_{N,p} = \oplus_C\bigl(\widetilde{R}_{Z_1^p}(X_1),\dots,\widetilde{R}_{Z_N^p}(X_N)\bigr)
\end{align}
when $p>0$. 
Note that the sets defined in \eqref{e:stoch_model_convexhull_p<=0} and \eqref{e:stoch_model_convexhullp>0} (and hence the stochastic model built upon them) can be recovered by setting $C = \conv\{e_1,\dots,e_N\}$.

The proofs of both Theorem \ref{t:stoch_iso_quermass} and \ref{t:stoc_affine_sobolev} essentially rely on checking that the corresponding functionals satisfy the conditions of Theorem \ref{t:Christ_BL}. In this regard, for the reader's convenience, we split the first part of Theorem \ref{t:stoch_iso_quermass} proof's in two lemmas. We start by adapting \cite[Proposition~5.1]{PRB2} to the context of level sets.
\begin{lemma}\label{l:level_sets_lps}
    Let $N\geq n+1$, $\{(x_i,z_i)\}_{i=1}^N\subset (\R^n\times\R_+)^N$ and $\{\lambda_i\}_{i=1}^N\subset \R$. Let $C\subset \R^N$ be a compact convex set contained in the positive orthant,
    $$
    \rho_i(p) = \begin{cases} 
      z_i^p& \mathrm{if}\quad p\neq 0, \\
     -\log z_i & \mathrm{if}\quad p=0
   \end{cases}
    $$
    and 
    \begin{align*}
         &E_{N,p}(t) = \oplus_C\bigl(\{R_{\rho_i(p)}(x_i + \lambda_i t \theta)\}_{i=1}^N\bigr) \quad (p\leq 0),\\
         &H_{N,p}(t)= \oplus_C\bigl(\{\widetilde{R}_{\rho_i(p)}(x_i + \lambda_i t \theta)\}_{i=1}^N\bigr) \quad (p > 0).
    \end{align*}
    Then, for any $z\in \R_+$, both $\{P_{e_{n+1}^\perp}(E_{p,N}(t)  \cap \pi_z)\}_{t\in I}$ and $\{P_{e_{n+1}^\perp}(H_{p,N}(t)\cap \pi_z)\}_{t\in I}$, where $\pi_z = e_{n+1}^\perp + z e_{n+1}$ and $I\subset \R$ an interval, are linear parameter systems. 
\end{lemma}
\begin{proof}
    We provide a sketch of the proof of the case $p\leq 0$ as the one of $p>0$ is analogous. First, note that
    \begin{align*}
        E_{N,p}(t)  &= \left\{\sum_{i=1}^N c_i (x_i + r_ie_{n+1}) + \left(\sum_{i=1}^Nc_i\lambda_i\right)t\theta: c \in C, \rho_i(p) \leq r_i\right\}\\
        &=\left\{(x,r): x=\sum_{i=1}^N c_i(x_i + \lambda_i t\theta), r \geq \sum_{i=1}^N c_i\rho_i(p), c\in C\right\}
    \end{align*}
    Let now $(r_1,\dots,r_N)\in\R^N$ and $g:C\subset\R^N \to \R$ be the function given by $g(c) = \esc{(r_1,\dots r_N),c}$. Fixed $z\in\R$, the set $\widetilde{C} = g^{-1}(\{z\})$ is a compact convex set in $\R^N$. Moreover, 
    $$
    E_{p,N}(t)  \cap \pi_z =  \{x_{\tilde{c}} + z e_{n+1} + \mu_{\tilde{c}} t\theta: \tilde{c}\in \widetilde{C}\} 
    $$
    for all $ z\geq \min_{c\in C}\sum_{i=1}^N c_i\rho_i(p)$.
  
    Thus, for any $t\in I$ and $ z\geq \min_{c\in C}\sum_{i=1}^N c_i\rho_i(p)$, $P_{e_{n+1}^\perp}(E_{p,N}(t) \cap \pi_z) = \{x_{\tilde{c}} + \mu_{\tilde{c}} t\theta: \tilde{c}\in \widetilde{C}\}$ is a linear parameter system indexed in $\widetilde{C}$. 
\end{proof}

The latter together with Corollary \ref{c:shepard_convexity} yields the following. \begin{lemma}\label{l:epi_is_lps}
     Let $N\geq n+1$ and $\{(x_i,z_i)\}_{i=1}^N\subset(\R^n\times\R_+)^N$. Then, for fixed $z_1,\dots z_N \in\R_+$, the function $F_{u,y};\R^N \to \R_+$ given by
     $$
     F_{u,y} (s_1,\dots,s_N) = \int_0^{+\infty} W_i(\{\Phi^N_{(y_k +s_ku,z_k)} \geq t\})\,\dlat t
     $$
     is convex and even for all $u\in \s^{n-1}$ and $y =(y_1,\dots,y_N)\in (u^\perp)^N$.
\end{lemma}

\begin{proof}
    For the reader's convenience we will use the notation $\bar s =(s_1,\dots,s_N)$. Let $\bar s,\bar s^\prime$ and $\lambda \in (0,1)$. As before, we only give a proof of the case $p\leq 0$; note that
    \begin{align*}
          F_{u,y} \bigl((1-\lambda)\bar s + \lambda \bar s^\prime\bigr) &= \int_0^{+\infty} W_i(\{\Phi^N_{(y_k + ((1-\lambda) s_k + \lambda s_k^\prime) u,z_k)} \geq t\})\,\dlat t\\
          &=\int_0^{+\infty} W_i\bigl(P_{e_{n+1}^\perp}(E_{p,N}(\lambda) \cap \pi_t)\bigr)\,\dlat t,
    \end{align*}
    where
    \begin{align}
        E_{p,N}(\lambda) &=  \oplus_C\bigl(\{R_{\rho_i(p)}\bigl(y_i + (1-\lambda) s_i + \lambda s_i^\prime\bigr) \theta\bigr)\}_{i=1}^N\bigr)\\
        &= \oplus_C\bigl(\{R_{\rho_i(p)}\bigl(y_i + s_i u + \lambda (s_i^\prime - s_i) u\bigr)\}_{i=1}^N\bigr)\\
        &= \oplus_C\bigl(\{R_{\rho_i(p)}(x_i + \mu_i \lambda u)\}_{i=1}^N\bigr),
    \end{align}
    and $\mu_i = s_i^\prime - s_i$ for all $i=1,\dots,N$. Hence, using Lemma \ref{l:level_sets_lps} we get that $\{P_{e_{n+1}^\perp}(E_{p,N}(\lambda) \cap \pi_t)\}_{\lambda \in (0,1)}$ is a linear parameter system. Thus, Corollary \ref{c:shepard_convexity} imply that $\lambda \mapsto  F_{u,y} \bigl((1-\lambda)\bar s + \lambda \bar s^\prime\bigr)$ is convex proving that $F_{u,y}(\bar s)$ is so. In addition, taking into account that
    $$
    \oplus_C\bigl(\{R_{\rho_i(p)}(y_i - s_iu)\}_{i=1}^N\bigr) = R_u\Bigl(\oplus_C\bigl(\{R_{\rho_i(p)}(y_i + s_iu)\}_{i=1}^N\bigr) \Bigr)
    $$
    for every $C\subset \R^N$, it is straightforward to check that 
    $$
    \{\Phi^N_{(y_k -s_ku,z_k)} \geq t\} = R_u\{\Phi^N_{(y_k+s_ku, z_k)}\geq t\}
    $$
    for all $t\in \R_+$. Therefore, $F_{u,y}(-\bar s) = F_{u,y}(\bar s )$ for all $u \in \s^{n-1}$ and $y \in (u^\perp)^N$. 
\end{proof}
We can now finish the proof of our first main result.
\begin{proof}[Proof of Theorem \ref{t:stoch_iso_quermass}]
   We will use for short the notation $\bar z =(z_1,\dots,z_N)\in(\R_+)^N$. Let $\{(X_k,Z_k)\}_{k=1}^N$ be independent random vectors distributed w.r.t. $f$. Note that using Fubini's theorem
\begin{align*}
    \E&\Bigl[W_i\Bigl(\Phi^N_{(X_k,Z_k)}\Bigr)\Bigr]\\
    &= \left(\int f\right)^{-N} \int_{(\R^n)^N}\int_{(\R_+)^N}W_i\Bigl(\Phi^N_{(x_k,z_k)}\Bigr) \prod_{i=1}^N \chi_{_{\{f \geq z_i\}}}(x_i)\, \dlat \bar z\,\dlat  x  \\
    &=  \left(\int f\right)^{-N} \int_{(\R_+)^N}\int_{(\R^n)^N} \left(\int_0^{+\infty} W_i(\{\Phi^N_{(x_k,z_k)}\geq t\})\,\dlat t \right) \prod_{i=1}^N \chi_{_{\{f \geq z_i\}}}(x_i)\,\dlat  x \, \dlat \bar z\\
    &= \left(\int f\right)^{-N} \int_{(\R_+)^N}\int_{(\R^n)^N} F_{\bar z} (x_1,\dots,x_N) \prod_{i=1}^N g_i(x_i) \,\dlat  x \, \dlat \bar z,
\end{align*}
where, fixed $z \in (\R_+)^N$, $F_{\bar z} : (\R^n)^N \to \R_+$ is the function given by 
$$
F_{\bar z}(x_1,\dots,x_N) = \int_0^{+\infty} W_i(\{\Phi^N_{(x_k,z_k)}\geq t\})\,\dlat t
$$
and $g_i(x_i) = \chi_{_{\{f \geq z_i\}}}(x_i)$ for all $i=1,\dots,N$. Hence, Lemma \ref{l:epi_is_lps} implies that $F_{\bar z}$ satisfies the conditions of Theorem \ref{t:Christ_BL} (those of the quasi-convex analogue) for all $\bar z \in (\R_+)^N$. Thus,   
\begin{align*}
    \int_{(\R^n)^N} F_{\bar z} (x_1,\dots,x_N) \prod_{i=1}^N g_i(x_i) \,\dlat \bar x  &\geq \int_{(\R^n)^N} F_{\bar z} (x_1,\dots,x_N) \prod_{i=1}^N g^*_i(x_i) \,\dlat x \\
    &= \int_{(\R^n)^N} F_{\bar z} (x_1,\dots,x_N) \prod_{i=1}^N \chi_{_{\{f \geq z_i\}^*}}(x_i) \,\dlat  x,
\end{align*}
for any $\bar z \in (\R_+)^N$. Putting all together, 
\begin{align*}
 \E&\left[W_i\Bigl(\Phi^N_{(X_k,Z_k)}\Bigr)\right]\\
 &= \left(\int f\right)^{-N}\int_{(\R^n)^N}\int_{(\R_+)^N}F_{\bar z}(x_1,\dots,x_N) \prod_{i=1}^N \chi_{_{\{f \geq z_i\}}}(x_i)\, \dlat \bar z\,\dlat  x  \\
    &\geq \left(\int f^*\right)^{-N}\int_{(\R^n)^N}\int_{(\R_+)^N}F_{\bar z}(x_1,\dots,x_N) \prod_{i=1}^N \chi_{_{\{f \geq z_i\}}}^*(x_i)\, \dlat \bar z \,\dlat x  \\
    &= \E\left[W_i\Bigl(\Phi^N_{(X^*_k,Z^*_k)}\Bigr)\right]
\end{align*}
as we wanted to prove.
\end{proof}

\begin{remark}
    Let $f_1,\dots,f_n : \R^n \to \R_+$ be a family of $p$-concave functions, with $p\in \R\cup\{\pm \infty\}$.  It was shown in \cite{MilRot13} that
    $$
    V(f_1,\dots,f_n) \geq V(f^*_1,\dots,f_n^*)
    $$
    generalizing, therefore, Theorem \ref{t:func_iso_quermass}. A stochastic analogue of the latter can be deduced by using the argument of Theorem \ref{t:stoch_iso_quermass}'s proof.
\end{remark}

We now give a proof of Theorem \ref{t:stoc_affine_sobolev}.

\begin{proof}[Proof of Theorem \ref{t:stoc_affine_sobolev}]
    We assume, for the sake of simplicity, that $\|f\|_{\infty} = 1$. In addition, we use the notation $\bar z =(z_1,\dots,z_N)\in(\R_+)^N$ for short. First note that 
    \begin{align*}
        \E&\bigl[\nu\bigl(\pp \Phi^N_{(X_k,Z_k)}\bigr)\bigr]\\
        &= \left(\int f\right)^{-N} \int_{(\R^n)^N}\int_{(\R_+)^N} \nu\bigl(\pp \Phi^N_{(x_k,z_k)}\bigr) \prod_{i=1}^N \chi_{_{\{f \geq z_i\}}}(x_i) \, \dlat \bar z \,\dlat x \\
        &=\left(\int f\right)^{-N} \int_{(\R_+)^N}\int_{(\R^n)^N} F_{\bar z}(x_1,\dots,x_N) \prod_{i=1}^N g_i(x_i)\,\dlat  x \, \dlat \bar z,
    \end{align*}
   where, for any fixed $\bar z \in (\R_+)^N$, $F_{\bar z} : (\R^n)^N \to \R_+$ is the function given by
    \begin{align*}
        F_{\bar z}(x_1,\dots,x_N) = \nu\bigl(\pp \Phi^N_{(x_k,z_k)}\bigr),
    \end{align*}
    and $g_i(x_i) = \chi_{_{\{f \geq z_i\}}}(x_i)$ for all $i=1,\dots,N$. Hence, fixing $\bar z \in (\R_+)^N$ and $u\in \s^{n-1}$,
    \begin{align*}
        F&_{\bar z,y, u} (\bar s)\\
        &= \int_{u^\perp}\int_{\R} \bigl[w+ru \in \{\|w + ru\|_{\pp \Phi^N(y_i + s_i u_i,z_i)}\leq 1\}\bigr]\psi(w + r u)\,\dlat r \,\dlat w\\
        &=\int_{u^\perp}\int_{\R} \bigl[w + ru \in \{nV(\Phi^N(y_i + s_i u,z_i)[n-1],\chi_{_{[0,w+ru]}}) \leq 1\}\bigr]\psi(w + r u) \,\dlat r\, \dlat w\\
        &=\int_{u^\perp}\int_{\R} \chi_{_{\{g(\bar s,r) \leq 1\}}} \psi(w + r u)\,\dlat r \,\dlat w,\\
        &=\int_{u^\perp}\varphi_w(\bar s)\,\dlat w,
    \end{align*}
    where, when needed, we use $[x\in A]$ for $\chi_{_A}$ to make reading easier. Moreover, $\varphi_w : u^\perp \times(\R)^N \to \R_+$ is the function given by 
    $$
    \varphi_w(\bar s) = \int_{\R}\chi_{_{\{g(\bar s,r) \leq 1\}}} \psi(w + r u)\,\dlat r
    $$
    and $g:\R^N \times \R\to \R_+$ is defined as $g(\bar s,r) = nV(\Phi^N_{(y_i + s_i u,z_i)}[n-1],\chi_{_{[0,w+ru]}})$.
    
    Secondly, for any given $\lambda \in (0,1)$
    \begin{align*}
    V&(\Phi^N_{(y_i + ((1-\lambda) s_i + \lambda s_i^\prime) u,z_i)}[n-1],\chi_{_{[0,w+((1-\lambda)r + \lambda r^\prime)u]}})\\
    &= \int_0^{1} V\bigl(P_{e_{n+1}^\perp}\bigl(E_{p,N}(\lambda)\cap \pi_t\bigr)[n-1],[0,w+ ru + \lambda (r^\prime - r)u]\bigr) \,\dlat t.
    \end{align*}
    Moreover, note that $\{[0,w+\bigl((1-\lambda)r + \lambda r^\prime\bigr)u]\}_{\lambda \in (0,1)}$ is a linear parameter system in the direction $u\in \s^{n-1}$. Hence, Lemma \ref{l:level_sets_lps} together with Theorem \ref{t:convexity_Mixed_Volumes)}, imply that  $(\bar s,r) \to g(\bar s,r)$ is jointly convex. Thus, we have for every $w \in u^\perp$ that 
    $$
    (\bar s, r)\mapsto\chi_{_{\{g(\bar s,r) \leq 1\}}} \psi(w + r u)
    $$
    is $(-1/n)$-concave. As a consequence, Corollary \ref{c:marginals} implies that $\bar s \mapsto \varphi_w(\bar s)$ is $\alpha$-concave, with $\alpha = -1/(n-1)$ (and, in particular, quasi-concave), for every $w \in u ^\perp$. Furthermore, taking into account that 
     $$
    \{\Phi^N_{(y_i -s_iu,z_i)} \geq t\} = R_u\{\Phi^N_{(y_i+s_iu,z_i) }\geq t\}
    $$
    for all $t\in\R_+$, we have that $g(-\bar s,-r) = g(\bar s,r)$. The latter, together with the rotational invariance of $\psi$, yield that $\varphi_w (\bar s) =  \varphi_w (-\bar s)$  for any $w \in u^\perp$.
    
    Now, using Fubini's theorem
    \begin{align*}
        \E&\bigl[\nu\bigl(\pp \Phi^N_{(X_k,Z_k)}\bigr)\bigr] \\
        &= \left(\int f\right)^{-N}\int_{(\R_+)^N}\int_{(u^\perp)^{N+1}}\left(\int_{\R^N}\varphi_w(\bar s)\prod_{i=1}^N h_i(s_i) \,\dlat \bar s \right)\, \dlat w \, \dlat y \,\dlat \bar z,
    \end{align*}
    where $h_i(s_i) = g_i(y_i + s_iu)$ for all $i =1 ,\dots , N$. Thus, using \eqref{e:BBL_Steiner} we get that
    \begin{align*}
        \int_{\R^N}\varphi_w(\bar s)\prod_{i=1}^N h_i(s_i) \,\dlat \bar s \leq \int_{\R^N}\varphi_w(\bar s)\prod_{i=1}^N h_i^*(s_i) \,\dlat \bar s 
    \end{align*}
    for every $w \in u^\perp$ and $\bar z \in (\R_+)^N$. Therefore, rolling back
    \begin{align*}
       \E&\bigl[\nu\bigl(\pp \Phi^N_{(X_k,Z_k)}\bigr)\bigr]\\
       &\leq \left(\int f\right)^{-N} \int_{(\R_+)^N}\int_{(\R^n)^N} F_{\bar z}(x_1,\dots,x_N) \prod_{i=1}^N g_i^u(x_i)\,\dlat  x \, \dlat \bar z\\
       &=  \left(\int f^u\right)^{-N} \int_{(\R_+)^N}\int_{(\R^n)^N} F_{\bar z}(x_1,\dots,x_N) \prod_{i=1}^N g_i^u(x_i)\,\dlat  x \, \dlat \bar z.
    \end{align*}

    As mentioned in Section \ref{s:background}, for every measurable function $f:\R^n \to \R_+$ with compact support there exists a sequence of the form $f_0 = f$ and $f_{n+1} = f_{n}^u$ for some $u \in \s^{n-1}$, which converges in the $L_1$ norm to $f^*$. Therefore, proceeding as in \cite{BLL74}, one gets that
    $$
    \int_{(\R^n)^N} F_{\bar z}(x_1,\dots,x_N) \prod_{i=1}^N g_i^u(x_i)\,\dlat  x  \leq \int_{(\R^n)^N} F_{\bar z}(x_1,\dots,x_N) \prod_{i=1}^N g_i^*(x_i)\,\dlat  x.
    $$
    for every $\bar z \in (\R_+)^N$. The latter yields that
   $$
    \E\Bigl[\nu\bigl(\pp \Phi^N_{(X_k,Z_k)}\bigr)\Bigr] \leq \E\Bigl[\nu\bigl(\pp \Phi^N_{(X^*_k,Z^*_k)}\bigr)\Bigl],
    $$
    as we wanted to see. 
\end{proof}

\begin{remark}\label{r:geometrical}
    We finish this section by showing how Theorems \ref{t:stoch_iso_quermass} and \ref{t:stoc_affine_sobolev} recover their geometrical counterparts. Let $K\subset\R^n$ be a convex body and $\{(X_k,Z_k)\}_{k=1}^N$ be independent random vectors uniformly distributed w.r.t. $\text{hyp}(\chi_{_K}) = K\times[0,1]$. Since $\chi_{_K}$ is a $(+\infty)$-concave function, its stochastic model is given by 
\begin{align*}
    \Phi^N_{(X_k,Z_k)}(x) = \sup\{\lim_{p\to +\infty} z^{1/p} : (x,z) \in h_{N,+\infty} \} = 1
\end{align*}
for all $x \in  h_{N,+\infty} = \oplus_C\bigl(\{X_1\},\dots,\{X_N\}\bigr) = [X_1,\dots,X_N]C$ and zero otherwise. In other words, $\Phi^N_{(X_k,Z_k)} = \chi_{_{[X_1,\dots,X_N]C}}$. Thus, on the one hand, applying Theorem \ref{t:stoch_iso_quermass} to $\chi_{_K}$ we get that for every $i=1,\dots,n-1$
\begin{equation}\label{e:general_stoch_iso_quermass}
    \E\Bigl[W_i\bigl([X_1,\dots,X_N]C\bigr)\Bigr] \geq \E\Bigl[W_i\bigl([X^*_1,\dots,X^*_N]C\bigr)\Bigr].
\end{equation}
 Note that \eqref{e:general_stoch_iso_quermass} recovers \eqref{e:groemer_quermass} when $C = \conv\{e_1, \dots, e_N\}$. Although not explicitly proved, this generalization of \eqref{e:groemer_quermass} was pointed out in \cite[Remark~4.4]{PP12}.
 
On the other hand, using Theorem \ref{t:stoc_affine_sobolev} one recovers that
\begin{equation}\label{e:empirical_Petty}
    \E\Bigl[\nu\bigl(\pp[X_1,\dots,X_N]C\bigr)\Bigr] \leq \E\Bigl[\nu\bigl(\pp[X^*_1,\dots,X^*_N]C\bigl)\Bigr],
\end{equation}
which was proved in \cite[Theorem~1.2]{PPT}.
\end{remark}

\section{Recovering deterministic results}\label{s:deterministic}
This section is devoted to study how to derive results for deterministic functions from their stochastic models. In this regard, it is enough for our purposes to consider the notion of \emph{epi-convergence}. 

Let $\{f_n\}_{n\in\N}$ be a sequence of functions. We say that $\{f_n\}_{n\in\N}$ epi-converges to $f$ or simply $e-\lim_{n\to +\infty} f_n = f$ if and only if 
$$
\mathrm{epi}(f_n) \overset{\mathrm{P-K}}{\to}\mathrm{epi}(f),
$$
where P-K denotes convergence in the usual Painlevé-Kuratowski sense (see e.g. \cite[Chapter~4]{RW97}). Analogously, we say that  $\{f_n\}_{n\in\N}$ \emph{hypo-converges} to $f$ (or $h-\lim_n f_n = f$) if and only if 
$$
\mathrm{hyp}(f_n) \overset{\mathrm{P-K}}{\to} \mathrm{hyp}(f).
$$ 

It is worth mentioning that convergence in the Hausdorff distance is stronger than convergence in the P-K sense and, when dealing with compact sets, both notions are equivalent (see e.g. \cite[p.~117]{RW97}). In a more detailed manner, let $\{A_n\}_{n\in\N}$ be a sequence of sets. If 
$$
A_n \overset{\delta^H}{\to} A,
$$
then $A_n \overset{\mathrm{P-K}}{\to}\,\overline{A}$, where $\overline{\cdot}$ denotes the closure of a given set. Since we work with closed epigraphs and hypographs  convergence in the Hausdorff distance will always imply convergence in the P-K sense.

We would like to remark that epi-convergence has been widely used in the literature, especially in the study of valuations on convex functions (see e.g. \cite{CLM17_2,CLM19,FM21,CLM24} and the references therein). We will use the following lemma originally proved in \cite[Lemma~5]{CLM19}:
\begin{lemma}\label{l:convergence_convex}
    Let $u:\R^n  \to \R \cup \{+\infty\}$ be a lower semi-continuous convex function such that $\lim_{|x| \to +\infty} u(x) = +\infty$. If $u_n : \R^n \to \R \cup \{+\infty\}$ is sequence of lower semi-continuous convex functions, with  $\lim_{|x| \to +\infty} u_n(x) = +\infty$ for all $n\in \N$, satisfying that $e-\lim_n u_n = u$. Then
    $$
    \{u_n \leq t\} \overset{\delta^H}{\to} \{u\leq t\}
    $$
    for all $t\neq\min_{x\in \R^n} u(x)$.
\end{lemma}
Lemma \ref{l:convergence_convex} immediately gives the following.
\begin{lemma}\label{l:convergence_level_sets}
    Let $f:\R^n \to \R_+$ be an integrable $p$-concave function, with $p\in \R\cup\{+\infty\}$ and $\{f_n\}_{n\in \N}$ a sequence of functions within the same class.
    \begin{enumerate}
        \item If $p> 0$ and $h-\lim_n f^p_n = f^p$, then 
        $$
        \{f_n \geq t\} \overset{\delta^H}{\to} \{f\geq t\}
        $$
        for all $t\neq \max_{x\in\R^n}f(x)$.
        
        \item If $p< 0$ and $e-\lim_n f^p_n = f^p$, then 
        $$
        \{f_n \geq t\} \overset{\delta^H}{\to} \{f\geq t\}
        $$
        for all $t\neq \max_{x\in\R^n}f(x)$.
        
        \item If $p = 0$ and $e-\lim_n (-\log f_n) = -\log f$, then 
        $$
        \{f_n \geq t\} \overset{\delta^H}{\to} \{f\geq t\}
        $$
        for all $t\neq \max_{x\in\R^n}f(x)$.
    \end{enumerate}
\end{lemma}

We will use the stochastic models presented in \eqref{e:stoch_model_convexhull_p<=0} and \eqref{e:stoch_model_convexhullp>0} to see how Theorems \ref{t:stoch_iso_quermass} and \ref{t:stoc_affine_sobolev} recover their deterministic counterparts. Let $N\geq n+1$, $f:\R^n \to \R_+$ be an integrable $p$-concave function and $\{(X_k,Z_k)\}_{k=1}^N$ be independent random vectors distributed w.r.t. f. Recall that the stochastic model of $f$ is built upon the random sets $E_{N,p},H_{N,p}\subset\R^n\times\R_+$ given by \begin{equation}
    E_{N,p} =  \begin{cases} 
      \conv\{R_{Z_1^p}(X_1),\dots,R_{Z_N^p}(X_N)\} & \mathrm{if}\quad p< 0, \\
       \conv\{R_{-\log Z_1}(X_1),\dots,R_{-\log Z_N}(X_N)\}& \mathrm{if}\quad p=0 
   \end{cases}
\end{equation}
and
\begin{equation}
     H_{N,p} = \conv\{\widetilde{R}_{Z_1^p}(X_1),\dots,\widetilde{R}_{Z_N^p}(X_N)\}
\end{equation}
if $p>0$.

Moreover, for a given integrable $p$-concave function $f:\R^n \to \R_+$, we will consider the truncated function $f_\varepsilon (x) = f(x)\chi_{_{\{f \geq \varepsilon\} }}(x)$. Note that $f_\varepsilon \leq f$ and dominated convergence, together with the continuity of the classical quermassintegrals in Hausdorff distance, imply that 
$$
\lim_{\varepsilon \to 0^+} W_i(f_\varepsilon) =W_i(f)
$$
for every $i\in\{0,1,\dots,n-1\}$. In addition, since any rotational invariant convex measure on $\R^n$ is continuous in Hausdorff distance (see e.g. \cite[Lemma~5.2]{CEFPP15}), one has that
$$
\lim_{\varepsilon\to0^+} \nu(\pp f_\varepsilon) = \nu(\pp f),
$$
where we have used that the functional polar projection body is continuous (see \cite[Lemma~4.1]{CLM17_2}). Therefore it is enough to prove our claims for $f_\varepsilon$.

Furthermore, let $\{(X_k,Z_k)\}_{k=1}^N$ be independent random vectors distributed w.r.t. $f_\varepsilon$. It is straightforward to check that the random sets $E_{N,p}$ and $H_{N,p}$ are convex bodies. Thus, using for instance Kolmogorov's $0-1$ law and Borel-
Cantelli lemma, one can check that
\begin{align*}
     &E_{N,p} \overset{\delta^H}{\to} \mathrm{epi}(f_\varepsilon^p) \quad \mathrm{if}\quad p< 0,\\
     &E_{N,p}  \overset{\delta^H}{\to} \mathrm{epi}(-\log f_\varepsilon) \quad \mathrm{if} \quad p = 0 \quad \mathrm{and}\\
     &H_{N,p}  \overset{\delta^H}{\to} \mathrm{hyp}(f_\varepsilon^p) \quad \mathrm{if} \quad p > 0
\end{align*}
almost surely as $N\to +\infty$.

We start by proving that almost sure (a. s. for short) convergence holds when considering the functional quermassintegral of a stochastic model. Although not explicitly described, we work in an underlying probability space.
\begin{proposition}\label{p:convergence_quermass}
    Let $f:\R^n \to \R_+$ be an integrable $p$-concave function and   $\{(X_k,Z_k)\}_{k=1}^N$ be independent random vectors distributed w.r.t. $f$. Then 
    $$
    \lim_{N\to\infty}W_i\bigl(\Phi^N_{(X_k,Z_k)}\bigr) = W_i(f)
    $$
    a. s. for every $i\in\{0,1,\dots,n-1\}$.
\end{proposition}
\begin{proof}
We will focus on the case of $p = 0$ as the rest is analogous. Let $N\geq n+1$ and $\{(X_k,Z_k)\}_{k=1}^{N}$ be independent random vectors distributed w.r.t $f_\varepsilon$. We have that $E_{N,p} \subset\R^{n+1}$ is a convex body and that 
$$
E_{N,p}  \overset{\delta^H}{\to} epi(-\log f_\varepsilon)
$$
a. s. when $N \to +\infty$. The latter, together with Lemma \ref{l:convergence_level_sets}, yield that 
$$
\{\Phi^N_{(X_k,Z_k)} \geq t\}  \overset{\delta^H}{\to} \{f_\varepsilon\geq t\}
$$
a. s. for all $t\neq \max_{x\in \R^n}f_\varepsilon(x)$ as $N\to +\infty$. Furthermore, the continuity of the classical quermassintegrals in Hausdorff distance implies that for all $t\neq \max_{x\in \R^n}f_\varepsilon(x)$ and $i\in\{0,1,\dots,n-1\}$
$$
\lim_{N\to+\infty} W_i\bigl(\{\Phi^N_{(X_k,Z_k)} \geq t\} \bigr) = W_i\bigl( \{f_\varepsilon\geq t\}\bigr)
$$
a. s. Finally, taking into account that $\{\Phi^N_{(X_k,Z_k)} \geq t\} \subset \{f_\varepsilon\geq t\} $
for all $t\in \R_+$, dominated convergence ensure that
$$
\lim_{N\to+\infty} W_i\bigl(\Phi^N_{(X_k,Z_k)}\bigr) = W_i(f_\varepsilon)
$$
a. s.
\end{proof}

We can now see how Theorem \ref{t:stoch_iso_quermass} recovers Theorem \ref{t:func_iso_quermass}. First, Proposition \ref{p:convergence_quermass} and dominated convergence ensure that
\begin{align*}
    W_i(f_\varepsilon) = \E&\Bigr[\lim_{N\to +\infty}W_i\bigl(\Phi^N_{(X_k,Z_k)}\bigr)\Bigl]\\
    &=\lim_{N\to +\infty}\E\Bigl[W_i\bigl(\Phi^N_{(X_k,Z_k)}\bigr)\Bigr]\\
\end{align*}
The latter, together with Theorem \ref{t:stoch_iso_quermass}, imply that
\begin{align*}
     W_i(f_\varepsilon) &= \lim_{N\to +\infty}\E\Bigl[W_i\bigl(\Phi^N_{(X_k,Z_k)}\bigr)\Bigr]\\
     &\geq \lim_{N\to +\infty}\E\Bigl[W_i\bigl(\Phi^N_{(X^*_k,Z^*_k)}\bigr)\Bigr]\\
    &= \E\Bigr[\lim_{N\to +\infty}W_i\bigl(\Phi^N_{(X^*_k,Z^*_k)}\bigr)\Bigl] = W_i(f_{\varepsilon}^*).
\end{align*}
Finally, as aforementioned, we recover Theorem \ref{t:func_iso_quermass} by taking limits as $\varepsilon \to 0^+$.

We now move on seeing how Theorem \ref{t:stoc_affine_sobolev} recovers its deterministic counterpart. First we prove the following.
\begin{proposition}\label{p:convergence_polar_projection_body}
     Let $f:\R^n \to \R_+$ be an integrable $p$-concave function and $\{(X_k,Z_k)\}_{k=1}^N$ be independent random vectors distributed w.r.t. $f$. Then 
    $$
    \pp \Phi_{(X_k,Z_k)}^N \overset{\delta^H}{\to} \pp f
    $$
    a. s. when $N\to +\infty$.
\end{proposition}
\begin{proof}
    Let $N\geq n+1$ and $\{(X_k,Z_k)\}_{i=1}^{N}$ be independent random vectors distributed w.r.t $f_\varepsilon$. As before, we will focus on the case of $p=0$. Moreover, since polarity is continuous w.r.t. the Hausdorff distance for convergence of centrally symmetric convex bodies, the claim is equivalent to check that 
    $$
\p \Phi^N_{(X_k,Z_k)} \overset{\delta^H}{\to} \p f
$$ 
a. s. when $N\to +\infty$.  Recall that 
$$
h_{\p f}(u) = n\int_0^{+\infty} V\bigl(\{f\geq t\}[n-1],[0,u]\bigr)\,\dlat t.
$$
We have, for all $t\neq \max_{x\in \R^n}f(x)$, that 
$$
\{\Phi^N_{(X_k,Z_k)} \geq t\}  \overset{\delta^H}{\to} \{f_\varepsilon\geq t\}
$$
a. s. when $N\to +\infty$. Hence, taking into account that  $\{\Phi^N_{(X_k,Z_k)} \geq t\} \subset \{f_\varepsilon\geq t\} $ for any $t\in \R_+$, the continuity of mixed volumes in Hausdorff distance and dominated convergence yield that 
$$
\lim_{N\to +\infty} \|h_{\p f_\varepsilon} -h_{\Phi^{N}_{(X_k,Z_k)}}\|_{\infty} = 0
$$
a. s., which is equivalent to the claim. 
\end{proof}

We finish as follows; Proposition \ref{p:convergence_polar_projection_body} and dominated convergence theorem ensure that 
\begin{align*}
    \nu(\pp f_\varepsilon) &= \E\Bigl[\nu \bigl(\lim_{N\to+\infty} \pp \Phi^N_{(X_k,Z_k)}\bigr)\Bigr]\\
    &= \E\Bigl[ \lim_{N\to+\infty}\nu \bigl(\pp \Phi^N_{(X_k,Z_k)}\bigr)\Bigl]\\
    &=\lim_{N\to+\infty}\E\Bigl[\nu\bigl(\pp\Phi^N_{(X_k,Z_k)}\bigr)\Bigl].\\
    \end{align*}
Hence, we deduce from Theorem \ref{t:stoc_affine_sobolev} that
    \begin{align*}
    \nu(\pp f_\varepsilon) &=\lim_{N\to+\infty}\E\Bigl[\nu\bigl(\pp\Phi^N_{(X_k,Z_k)}\bigr)\Bigl]\\
    &\leq  \lim_{N\to+\infty} \E\Bigl[\nu\bigl(\pp\Phi^N_{(X^*_k,Z^*_k)}\bigr)\Bigl]\\
    &=\E\Bigl[\nu \bigl(\lim_{N\to+\infty} \pp \Phi^N_{(X^*_k,Z^*_k)}\bigr)\Bigr] = \nu(\pp f_{\varepsilon}^*),
\end{align*}
which implies the claim after taking limits as $\varepsilon \to 0^+$.
\bigskip

{\bf Acknowledgments:}  We would like to thank Peter Pivovarov for helpful correspondence, and Fabian Mussnig for valuable comments.  We also extend our sincere gratitude to the anonymous referee, whose remarks have greatly enhanced the clarity and quality of our manuscript.
\bigskip

{\bf Conflicts of Interest:} There are no conflicts of interest.
\bibliographystyle{acm}
\bibliography{references_thesis}
\end{document}